\documentclass[reqno]{amsart}

\usepackage{amssymb}
\usepackage{amsfonts}
\usepackage{mathrsfs}
\usepackage{bbm}

\usepackage{xcolor}

\theoremstyle{plain}
    \newtheorem{theorem}{Theorem}[section]
    \newtheorem*{theorem*}{Theorem}
    \newtheorem{lemma}[theorem]{Lemma}
    \newtheorem{proposition}[theorem]{Proposition}
    \newtheorem{corollary}[theorem]{Corollary}

\theoremstyle{definition}
    
    \newtheorem{remark}{Remark}[section]
    
\numberwithin{equation}{section}

\renewcommand{\r}{\right}
\newcommand{\cleq}{\lesssim}

\newcommand{\ceq}{\approx} 
\newcommand{\eps}{\varepsilon}
\DeclareMathOperator{\im}{Im}
\DeclareMathOperator{\re}{Re}

\DeclareMathOperator{\NLS}{NLS}

\begin{document}

\title[]{Blow-up or Grow-up for the threshold solutions to the nonlinear Schr\"{o}dinger equation}
\author[S. Gustafson]{Stephen Gustafson}
\address[S. Gustafson]{University of British Columbia, 1984 Mathematics Rd., Vancouver, Canada V6T1Z2.}
\email{gustaf@math.ubc.ca}
\author[T. Inui]{Takahisa Inui}
\address[T. Inui]{Department of Mathematics, Graduate School of Science, Osaka University, Toyonaka, Osaka, Japan 560-0043.
\newline 
University of British Columbia, 1984 Mathematics Rd., Vancouver, Canada V6T1Z2.}
\email{inui@math.sci.osaka-u.ac.jp 
}
\date{\today}
\date{\today}
\maketitle

\begin{abstract}
We consider the nonlinear Schr\"{o}dinger equation with $L^{2}$-supercritical and $H^{1}$-subcritical power type nonlinearity. Duyckaerts and Roudenko \cite{DuRo10} and Campos, Farah, and Roudenko \cite{CFR20} studied the global dynamics of the solutions with same mass and energy as that of the ground state. In these papers, finite variance is assumed to show the finite time blow-up. In the present paper, we remove the finite-variance assumption and prove a blow-up or grow-up result. 
\end{abstract}

\tableofcontents

\section{Introduction}

\subsection{Background} 

We consider the following nonlinear Schr\"{o}dinger equation:
\begin{align}
\label{NLS}
\tag{NLS}
	\left\{
	\begin{aligned}
		&i \partial_{t}u + \Delta u +|u|^{p-1}u=0, & (t,x) \in I \times \mathbb{R}^{d},
		\\
		&u(0,x)=u_{0}(x), &x \in \mathbb{R}^{d},
	\end{aligned}
	\r.
\end{align}
where $d \in \mathbb{N}$, $1+4/d < p < (d+2)/(d-2)$. Here, we regard $(d+2)/(d-2)$ as $\infty$ when $d=1,2$. The nonlinearity is called $L^{2}$-supercritical and $H^{1}$-subcritical. The initial value problem is locally well-posed in $H^{1}(\mathbb{R}^{d})$ (see \cite{GiVe79} and the standard text books \cite{Caz03,Tao06,LiPo15}). We also know that the blow-up criterion holds. That is, if the forward (resp. backward) maximal existence time $T_{\max}$ (resp. $T_{\min}$) is finite then $\|\nabla u(t)\|_{L^{2}}$ diverges at $t=T_{\max}$ (resp. $t=T_{\min}$). Moreover, the mass $M$, energy $E$, and momentum $P$,
given by 
\begin{align*}
	M(u)&=\|u\|_{L^{2}}^{2},
	\\
	E(u)&=\frac{1}{2}\|\nabla u\|_{L^{2}}^{2} - \frac{1}{p+1}\|u\|_{L^{p+1}}^{p+1},
	\\
	P(u)&=\im \int_{\mathbb{R}^{d}} \overline{u(x)} \nabla u(x) dx
\end{align*}
are conserved by the flow.
In the study of the global dynamics, the ground state solution $e^{it}Q(x)$ plays a crucial role. Here, $Q$ is the least-energy radial, positive solution to the elliptic equation:
\begin{align*}
	-\Delta Q +Q -Q^{p}=0 \text{ in } \mathbb{R}^{d}.
\end{align*}
Since pioneering work by Kenig and Merle \cite{KeMe06}, many researchers have studied \eqref{NLS} from the viewpoint of global dynamics. 
We summarize the global dynamics of solutions below the ground state:
\begin{theorem*}[Global dynamics below the ground state]
Let $u_{0}\in H^{1}(\mathbb{R}^{d})$ satisfy $M(u_{0})^{\frac{1-s_{c}}{s_{c}}}E(u_{0})<M(Q)^{\frac{1-s_{c}}{s_{c}}}E(Q)$, where $s_{c}:=\frac{d}{2} -\frac{2}{p-1}$.
\begin{enumerate}
\item If $\|u_{0}\|_{L^{2}}^{1-s_{c}}\|\nabla u_{0}\|_{L^{2}}^{s_{c}} < \|Q\|_{L^{2}}^{1-s_{c}}\|\nabla Q\|_{L^{2}}^{s_{c}}$, then the solution $u$ exists for $t \in \mathbb{R}$, and scatters in both time directions. 
\item If $\|u_{0}\|_{L^{2}}^{1-s_{c}}\|\nabla u_{0}\|_{L^{2}}^{s_{c}} > \|Q\|_{L^{2}}^{1-s_{c}}\|\nabla Q\|_{L^{2}}^{s_{c}}$, then the solution $u$ blows up or grows up in both time directions.
\end{enumerate}
\end{theorem*}

\begin{remark}
\begin{enumerate}
\item $\|u_{0}\|_{L^{2}}^{1-s_{c}}\|\nabla u_{0}\|_{L^{2}}^{s_{c}} = \|Q\|_{L^{2}}^{1-s_{c}}\|\nabla Q\|_{L^{2}}^{s_{c}}$ does not occur since we assume $M(u_{0})^{\frac{1-s_{c}}{s_{c}}}E(u_{0})<M(Q)^{\frac{1-s_{c}}{s_{c}}}E(Q)$. 
\item We say that $u$ grows up in positive time if $u$ exists at least on $[0,\infty)$ and $\limsup_{t \to \infty}\|\nabla u(t)\|_{L^{2}}=\infty$. 
\item In the second statement, we have four possibilities. For example, $u$ blows up in finite positive time and grows up in the negative time direction. 
\item Martel and Rapha\"{e}l \cite{MaRa18} showed existence of a grow-up solution for the $L^{2}$-critical NLS on $\mathbb{R}^{2}$ (i.e., $p=3$). However, it is still an open problem whether a grow-up solution to \eqref{NLS} exists. 
\end{enumerate}
\end{remark}

Holmer and Roudenko \cite{HoRo08} showed the scattering part for 3d cubic NLS in the radial setting. The radial assumption was removed by Duyckaerts, Holmer, and Roudenko \cite{DHR08}. The blow-up or grow-up part for 3d cubic NLS is obtained by Holmer and Roudenko \cite{HoRo10}. 
Akahori and Nawa \cite{AkNa13} and Guevara \cite{Gue14} proved both scattering and blow-up or grow-up results for \eqref{NLS} (i.e. for general power and dimension). Fang, Xie, and Cazenave \cite{FXC11} showed the scattering for \eqref{NLS}. Guo \cite{Guo16} obtained the blow-up or grow-up result for \eqref{NLS}. Their proofs of scattering rely on a concentration compactness and rigidity argument, which was developed by Kenig and Merle \cite{KeMe06}. Dodson and Murphy \cite{DoMu17} give another proof of the scattering part for 3d cubic NLS in the radial setting. They used an interaction Morawetz estimate and Tao's scattering criterion (\cite{Tao04}). They removed the radial assumption in \cite{DoMu18}. See also \cite{ADM20,GuSh20}. 

While the global behavior of solutions below the ground state is of only two types, i.e., scattering or blow-up (or grow-up), another global behavior appears in the threshold case -- there exist solutions converging to the ground state:
\begin{theorem*}[Existence of special solutions on the threshold]
There exist two radial solutions $Q^{\pm}(t)$ to \eqref{NLS} in $H^{1}(\mathbb{R}^{d})$ such that 
\begin{itemize}
\item $M(Q^{\pm})=M(Q)$, $E(Q^{\pm})=E(Q)$, $Q^{\pm}$ exist at least on $[0,\infty)$, and there exist $C,c>0$ such that 
\begin{align*}
	\|Q^{\pm}(t) - e^{it}Q\|_{H^{1}}\leq C e^{-ct} 
\end{align*}
for $t\geq 0$,
\item $\|\nabla Q^{+}(0)\|_{L^{2}} > \|\nabla Q\|_{L^{2}}$ and $Q^{+}$ blows up in finite negative time,
\item $\|\nabla Q^{-}(0)\|_{L^{2}} < \|\nabla Q\|_{L^{2}}$ and $Q^{-}$ exists globally and scatters backward in time.
\end{itemize}
\end{theorem*}
Moreover, we know the global dynamics on the threshold:
\begin{theorem*}[Global dynamics of threshold solutions]
Let $u_{0}\in H^{1}(\mathbb{R}^{d})$ satisfy $M(u_{0})^{\frac{1-s_{c}}{s_{c}}}E(u_{0})=M(Q)^{\frac{1-s_{c}}{s_{c}}}E(Q)$.
\begin{enumerate}
\item If $\|u_{0}\|_{L^{2}}^{1-s_{c}}\|\nabla u_{0}\|_{L^{2}}^{s_{c}} < \|Q\|_{L^{2}}^{1-s_{c}}\|\nabla Q\|_{L^{2}}^{s_{c}}$, then the solution $u$ exists on $\mathbb{R}$. Moreover, either $u$ scatters in both directions or $u=Q^{-}$ up to symmetries of the equation. 
\item If $\|u_{0}\|_{L^{2}}^{1-s_{c}}\|\nabla u_{0}\|_{L^{2}}^{s_{c}} = \|Q\|_{L^{2}}^{1-s_{c}}\|\nabla Q\|_{L^{2}}^{s_{c}}$, then $u=e^{it}Q$  up to symmetries of the equation. 
\item We assume that $u_{0}$ is radial when $d\geq 2$ or $|x|u_{0} \in L^{2}$. If $\|u_{0}\|_{L^{2}}^{1-s_{c}}\|\nabla u_{0}\|_{L^{2}}^{s_{c}} > \|Q\|_{L^{2}}^{1-s_{c}}\|\nabla Q\|_{L^{2}}^{s_{c}}$, then either the solution $u$ blows up in finite positive and negative time or $u=Q^{+}$ up to symmetries of the equation. 
\end{enumerate}
\end{theorem*}

\begin{remark}
The case $d=1$ with radial (even) symmetry is not 
included in part (3). 
\end{remark}

These result were shown by Duyckaerts and Roudenko \cite{DuRo10} for 3d cubic NLS. Recently, these were extended to general dimensions and powers by Campos, Farah, and Roudenko \cite{CFR20}. Global dynamics above the ground state were studied by Nakanishi and Schlag \cite{NaSc12}, but we do not pursue that 
direction here.

We are interested in the blow-up result on the threshold. While the blow-up or grow-up result without finite variance is known below the ground state, the result without finite variance (or radial symmetry) was not known at the threshold. In the present paper we remove this additional assumption.

\subsection{Main result}

We define the virial functional $K$ by
\begin{align*}
	K(\varphi):= \|\nabla \varphi\|_{L^{2}}^{2} - \frac{d(p-1)}{2(p+1)}\|\varphi\|_{L^{p+1}}^{p+1}.
\end{align*}

\begin{theorem}
\label{thm1.1}
Let $u$ be a solution to \eqref{NLS} satisfying
\begin{align}
\label{MEN}
\tag{MEN}
	M(u_{0})=M(Q), \; E(u_{0})=E(Q), \; K(u_{0})<0.
\end{align}
Then $u$ must satisfy one of the following:
\begin{enumerate}
\item $u$ blows up in both time directions,
\item $u$ blows up in positive time and grows up in negative time,
\item $u$ blows up in negative time and grows up in positive time,
\item $u$ grows up in both time directions,
\item $u=Q^{+}$ up to symmetries.
\end{enumerate}
\end{theorem}

\begin{remark}
The negativity of the virial functional is equivalent to $\|\nabla Q\|_{L^{2}}<\|u\|_{L^{2}}$ under the mass-energy condition $M(u)=M(Q), E(u)=E(Q)$. 
\end{remark}

By scaling and the variational structure, we get the following statement as a corollary of Theorem \ref{thm1.1} (see \cite[Remark 1.6]{CFR20}). 
\begin{corollary}
Let $u$ be the solution satisfying $M(u_{0})^{\frac{1-s_{c}}{s_{c}}}E(u_{0})=M(Q)^{\frac{1-s_{c}}{s_{c}}}E(Q)$ and $\|u_{0}\|_{L^{2}}^{1-s_{c}}\|\nabla u_{0}\|^{s_{c}}>\|Q\|_{L^{2}}^{1-s_{c}}\|\nabla Q\|^{s_{c}}$. Then the conclusion of Theorem \ref{thm1.1} holds. 
\end{corollary}

%
%


\subsection{Idea of the proof}
The virial identity
\begin{align*}
	\frac{d^{2}}{dt^{2}}\left(\int_{\mathbb{R}^{d}}|x|^{2}|u(t,x)|^{2}dx \right)
	=8K(u(t))
\end{align*}
is very useful for showing blow-up. Glassey \cite{Gla77} used it to show that solutions with negative energy blow up in finite time if the initial data is of finite variance. To remove the finite variance assumption in the use of the virial identity, a localized virial identity 
\begin{align*}
	\frac{d^{2}}{dt^{2}}\left(\int_{\mathbb{R}^{d}}\varphi_{R}(x)|u(t,x)|^{2}dx \right)
	=8K(u(t)) + A_{R}(u(t))
\end{align*}
is used, where $\varphi_{R} \in C_{0}^{\infty}$ satisfies $\varphi(x)=|x|^{2}$ for $|x|\leq R$ and $A_{R}$ is an error term. Under the radially symmetric assumption, by using this localized virial identity and the radial Sobolev inequality to control the error term, Ogawa and Tsutsumi \cite{OgTs91jde} proved blow-up of solutions with negative energy when $d\geq2$. Ogawa and Tsutsumi also showed blow-up for 1d quintic NLS without the radial assumption by using a localized virial identity and scaling argument in \cite{OgTs91pams}. Their proof relies on the $L^{2}$-critical nonlinearity. As far as the authors know, there is no finite time blow-up result for even solutions of \eqref{NLS} in the one dimensional case. 
Akahori and Nawa \cite{AkNa13} gave a proof of blow-up or grow-up under the assumption $\sup_{t>0} K(u(t))<0$, which holds under the mass-energy condition $M(u_{0})^{\frac{1-s_{c}}{s_{c}}}E(u_{0})<M(Q)^{\frac{1-s_{c}}{s_{c}}}E(Q)$ and negativity of $K(u_{0})$ (see also \cite{DWZ16}). They derived a contradiction by controlling the error term in the localized virial identity assuming the uniform boundedness $\sup_{t>0}\|\nabla u(t)\|_{L^{2}}<\infty$. However, this argument does not work on the threshold. Indeed, we have the possibility that $K(u(t))\to 0$ as $t\to \infty$. To overcome this difficulty, we use a concentration compactness argument, obtained by Keraani \cite{Ker01} and developed by Kenig and Merle \cite{KeMe06} to show the scattering result. The 
method of application of the concentration compactness argument to show blow-up was designed by Holmer and Roudenko \cite{HoRo10} in the study of 3d cubic NLS, and was extended to general dimensions and powers by Guevara \cite{Gue14} and Guo \cite{Guo16}. By the concentration compactness argument, we can show that the orbit of the global-in-time uniformly bounded solution $u$ with negative $K$ is compact in $H^{1}$. More precisely, there exists $X:(0,\infty) \to \mathbb{R}^{d}$ such that $\{u(t, \cdot-X(t)):t>0\}$ is precompact in $H^{1}(\mathbb{R}^{d})$. In the 
below-threshold case, we have $\sup_{t>0} K(u(t))<0$, the error term is controlled by the precompactness, and we get a contradiction. On the threshold, $K(u(t))\to 0$ may happen. We apply a modulation argument to show that $u(t) \to e^{i\theta_{0}}Q(\cdot-X_{0})$ as $t\to \infty$ in $H^{1}(\mathbb{R}^{d})$. This argument is similar to the case of positive $K$, done by Duyckaerts and Roudenko \cite{DuRo10} (see also \cite{CFR20}). Once we get the convergence, we can show that the uniformly bounded global solution $u$ is of finite variance by using negativity of $K$. The solution with finite variance must be $Q^{+}$ up to symmetries by \cite{DuRo10,CFR20}. This implies Theorem \ref{thm1.1}. 

\section{Preliminaries}

\subsection{Variational structure}

In this section, we collect lemmas concerning the variational structure. 

\begin{lemma}[Pohozaev identity (e.g. {\cite[Corollary 8.1.3]{Caz03}})]
We have
\begin{align*}
	\|Q\|_{L^{p+1}}^{p+1}=\frac{2(p+1)}{d(p-1)}\|\nabla Q\|_{L^{2}}^{2}=\frac{p+1}{p-1}\|Q\|_{L^{2}}^{2}.
\end{align*}
In particular, we have
\begin{align*}
	E(Q)
	=\frac{d(p-1)-4}{2d(p-1)}\|\nabla Q\|_{L^{2}}^{2}.
\end{align*}
\end{lemma}

The next lemma follows from the variational structure of $Q$. 

\begin{lemma}
Let $u_{0}$ satisfy $M(u_{0})=M(Q)$ and $E(u_{0})=E(Q)$ and $u$ be the solution to \eqref{NLS} with $u(0)=u_{0}$. Then we have the following.
\begin{enumerate}
\item If $K(u_{0})>0$, then $K(u(t))>0$ for all $t\in \mathbb{R}$.
\item If $K(u_{0})=0$, then $u_{0}(x)=e^{i\theta_{0}}Q(x-x_{0})$ for some $\theta_{0} \in \mathbb{R}$ and $x_{0} \in\mathbb{R}^{d}$.
\item If $K(u_{0})<0$, then $K(u(t))<0$ for all $t$ in the interval of existence. 
\end{enumerate}
\end{lemma}

See {\cite[Lemma 2.2]{DuRo10}} for the proof when $d=3$ and $p=3$.

%
%
%


\subsection{Localized virial identity}

For a solution $u$ to \eqref{NLS}, define 
\begin{align*}
	J_{R}(t)=\int_{\mathbb{R}^{d}} \varphi_{R}(x)|u(t,x)|^{2}dx,
\end{align*}
where $\varphi_{R}(x) = R^{2}\varphi\left(\frac{|x|}{R}\right)$ and $\varphi:[0,\infty) \to [0,\infty)$ satisfies 
\begin{align*}
	\varphi(r)=
	\begin{cases}
	r^{2}, & (r \leq 1),
	\\
	0, & (r\geq 3).
	\end{cases}
\end{align*}
and 
\begin{align*}
	\frac{d^{2}\varphi}{dr^{2}}(r) \leq 2 \text{ for } r \geq 0.
\end{align*}
Then by standard computations we have
\begin{align*}
	J_{R}'(t)&=2\im \int_{\mathbb{R}^{d}} \nabla \varphi_{R} \nabla u \overline{u}dx,
	\\
	J_{R}''(t)&=8K(u(t))+A_{R}(u(t)),
\end{align*}
where
\begin{align*}
	A_{R}(u(t))&=\int_{|x|\geq R} |\nabla u(t,x)|^{2}\left( 4\partial_{r}^{2}\varphi_{R}-8\right) dx
	\\
	&\quad +\frac{2(p-1)}{p+1} \int_{|x|\geq R} |u(t,x)|^{p+1}\left(2d-\Delta\varphi_{R} \right) dx
	+\int_{|x|\geq R} |u(t,x)|^{2}\Delta^{2}\varphi_{R} dx.
\end{align*}


\subsection{Strichartz norms}

Let $s\in [0,1)$. We say that a pair $(q,r) \in \mathbb{R}^{2}$ is $\dot{H}^{s}$-admissible and denote $(q,r) \in \Lambda_{s}$ if 
\begin{align*}
	\frac{2}{q} + \frac{d}{r} = \frac{d}{2} -s, \quad 
	\frac{2d}{d-2s} < r< \frac{2d}{d-2}.
\end{align*}
We define 
\begin{align*}
	\|u\|_{S(\dot{H}^{s};I)}:=\sup_{(q,r)\in \Lambda_{s}}\|u\|_{L_{t}^{q}(I;L_{x}^{r})}. 
\end{align*}
We set $\|u\|_{S(\dot{H}^{s})}=\|u\|_{S(\dot{H}^{s};\mathbb{R})}$ for short. 


\section{Modulation}
\label{sec3}

We define
\begin{align*}
	\mu(f):=\|\nabla Q\|_{L^{2}}^{2} - \|\nabla f\|_{L^{2}}^{2}.
\end{align*}
By the Pohozaev identity, if $E(u)=E(Q)$, we have 
\begin{align*}
	K(u)=\frac{d(p-1)-4}{4} \mu(u).
\end{align*}

See \cite{DuRo10,CFR20} for Lemmas \ref{lem3.2} and \ref{lem3.3} below. 

\begin{lemma}[Modulation]
\label{lem3.2}
There exist a positive constant $\mu_{0}$ and a function $\varepsilon(\mu): (0,\mu_{0}) \to (0,\infty)$ with $\varepsilon(\mu)\to 0$ as $\mu\to 0$ such that, for any $f\in H^{1}(\mathbb{R}^{d})$ satisfying $M(f)=M(Q)$, $E(f)=E(Q)$, and $|\mu(f)|<\mu_{0}$, there exist a $C^{1}$-mapping 
$f \mapsto (x,\theta)$ satisfying 
\begin{align*}
	&\|e^{-i\theta} f(\cdot +x) - Q\|_{H^{1}} \leq \varepsilon(\mu(f)),
	\\
	&\im \int_{\mathbb{R}^{d}} g(x) Q(x)dx =0
	\text{ and } \re \int_{\mathbb{R}^{d}} g(x) \nabla Q(x)dx=0,
\end{align*}
where $g=e^{-i\theta} f(\cdot +x) - Q$.
\end{lemma}

Let $u$ be a solution to \eqref{NLS} and $I_{\mu_{0}}:=\{t\in\mathbb{R}: |\mu(u(t))|<\mu_{0}\}$. Then by choosing $(x(t),\theta (t))$ for $t\in I_{\mu_{0}}$ as in Lemma \ref{lem3.2}, we can write
\begin{align*}
	u(t,x+x(t))=e^{i\theta(t)+it} (Q+g)
	=e^{i\theta(t)+it} (Q+\rho(t)Q+h),
\end{align*}
where
\begin{align*}
	\rho(t):= \frac{\re \int g Q^{p} dx}{\|Q\|_{L^{p+1}}^{p+1}}.
\end{align*}

\begin{lemma}[Estimates for the modulation parameters]
\label{lem3.3}
Let $u$ be a solution to \eqref{NLS} satisfying $M(u)=M(Q)$ and $E(u)=E(Q)$. The following estimates are valid for $t \in I_{\mu_{0}}$, taking $\mu_{0}$ smaller if necessary:
\begin{align*}
	&|\rho(t)| \approx \|h(t)\|_{H^{1}} \approx  \|g(t)\|_{H^{1}} \approx  \left| \re \int Qh  dx \right| \approx  |\mu(u(t))|,
	\\
	&|\rho'(t)| + |x'(t)| + |\theta'(t)| \lesssim |\mu(u(t))|. 
\end{align*}
\end{lemma}


\section{Profile decomposition}

We have the following nonlinear profile decomposition. See \cite[Proposition 3.6]{Gue14} and \cite[Proposition 6.5]{Guo16}

\begin{lemma}[Nonlinear profile decomposition]
\label{lem4.1}
Let $\{\varphi_{n}\}_{n \in \mathbb{N}}$ be a bounded sequence in $H^{1}(\mathbb{R}^{d})$. Then, there exists a subsequence, which is also denoted by $\{\varphi_{n}\}_{n \in \mathbb{N}}$, sequences $\{\psi^{j}\}_{j \in \mathbb{N}}$ and $\{W_{n}^{j}\}_{n,j\in \mathbb{N}}$ in $H^{1}(\mathbb{R}^{d})$, time sequence $\{t_{n}^{j}\}_{n,j\in \mathbb{N}}$ in $\mathbb{R}$, and spatial sequence $\{x_{n}^{j}\}_{n,j\in \mathbb{N}}$ in $\mathbb{R}^{n}$ such that 
\begin{align*}
	\varphi_{n}=\sum_{j=1}^{J} \NLS(-t_{n}^{j})\psi^{j}(\cdot-x_{n}^{j}) +W_{n}^{J}
\end{align*}
for each $J \in \mathbb{N}$, and the following statements hold:
\begin{enumerate}
\item For fixed $j$, $t_{n}^{j}=0$ for all $n$ or $|t_{n}^{j}| \to \infty$ as $n \to \infty$.
\item If $t_{n}^{j} \to \infty$, then $\|\NLS(-t)\psi^{j}\|_{S(\dot{H}^{s_{c}};[0,\infty))}< \infty$ and if $t_{n}^{j} \to -\infty$, then $\|\NLS(-t)\psi^{j}\|_{S(\dot{H}^{s_{c}};(-\infty,0])}< \infty$.
\item The orthogonality of the parameters: for $j \neq k$, 
\begin{align*}
	\lim_{n \to \infty} |t_{n}^{j}-t_{n}^{k}| + |x_{n}^{j}-x_{n}^{k}|=\infty.
\end{align*}
\item Smallness of the remainder: $\NLS(t)W_{n}^{J}$ is global for sufficiently large $J$ and
\begin{align*}
	\lim_{n \to \infty} \|\NLS(t)W_{n}^{J}\|_{S(\dot{H}^{s_{c}})} \to 0 \text{ as } J \to \infty.
\end{align*}
\item The orthogonality of the norms: For any $J\in \mathbb{N}$, 
\begin{align*}
	\|\varphi_{n}\|_{\dot{H}^{s}}^{2} = \sum_{j=1}^{J} \|\NLS(-t_{n}^{j})\psi^{j}\|_{\dot{H}^{s}}^{2}+ \|W_{n}^{J}\|_{\dot{H}^{s}}^{2} + o_{n}(1)
\end{align*}
for all $s \in [0,1]$ and 
\begin{align*}
	\|\varphi_{n}\|_{L^{p+1}}^{p+1} = \sum_{j=1}^{J} \|\NLS(-t_{n}^{j})\psi^{j}\|_{L^{p+1}}^{p+1}+ \|W_{n}^{J}\|_{L^{p+1}}^{p+1} + o_{n}(1).
\end{align*}
In particular, we have
\begin{align*}
	E(\varphi_{n})= \sum_{j=1}^{J}E(\psi^{j})+E(W_{n}^{J})+o_{n}(1)
\end{align*}
for any $J \in \mathbb{N}$. 
\end{enumerate}
\end{lemma}


We also have the following $\dot{H}^{1}$-orthogonality along the flow. See \cite[Lemma 3.9]{Gue14} and \cite[Lemma 6.7]{Guo16}

\begin{lemma}[$\dot{H}^{1}$ Pythagorean decomposition along the NLS flow]
Let $\{\varphi_{n}\}_{n \in \mathbb{N}}$ be a bounded sequence in $H^{1}(\mathbb{R}^{d})$. Fix time $T \in (0,\infty)$ arbitrarily. Assume that $u_{n}(t):=\NLS(t)\varphi_{n}$ exists on $[0,T]$ for all $n\in \mathbb{N}$ and 
\begin{align*}
	\sup_{n \in \mathbb{N}} \|\nabla u_{n}(t)\|_{L^{\infty}L^{2}[0,T]} < \infty.
\end{align*}
Let $W_{n}^{J}(t):=\NLS(t)W_{n}^{J}$, where $W_{n}^{J}$ is as in Lemma \ref{lem4.1} and $W_n^J(t)$ is global if $J,n$ are large. If $J$ is sufficiently large, then, for all $j$, $v^{j}(t-t_n^j):=\NLS(t-t_n^j)\psi^{j}$ exists on $[0,T]$ for large $n$ and 
\begin{align}
\label{eq4.7}
	\|\nabla u_{n}(t)\|_{L^{2}}^{2} =\sum_{j=1}^{J} \|\nabla v^{j}(t-t_{n}^{j})\|_{L^{2}}^{2} +\|\nabla W_{n}^{J}(t)\|_{L^{2}}^{2} +o_{n}(1)
\end{align}
for all $t \in [0,T]$, where $o_{n}(1) \to 0$ uniformly on $[0,T]$. We also have
\begin{align}
\label{eq4.8}
	\|\varphi_{n}\|_{L^{p+1}}^{p+1}
	=\sum_{j=1}^{J} \|v^{j}(t-t_{n}^{j})\|_{L^{p+1}}^{p+1}
	+\|W_{n}^{J}(t)\|_{L^{p+1}}^{p+1} + o_{n}(1) 
\end{align}
for $ t\in [0,T]$. 
\end{lemma}

%
%

The following lemma is essentially given by \cite[Lemma 4.7]{Gue14} and \cite[Lemma 6.7]{Guo16}. 

\begin{lemma}[Profile reordering]
\label{lem4.3}
Let $\{\varphi_{n}\}_{n \in \mathbb{N}}$ be a bounded sequence in $H^{1}(\mathbb{R}^{d})$. Assume that $M(\varphi_{n})=M(Q)$, $E(\varphi_{n})=E(Q)$, and $\|\nabla \varphi_{n}\|_{L^{2}}/\|\nabla Q\|_{L^{2}} >1$ for all $n$. Then the profiles can be reordered such that there exist $1\leq J_{1} \leq J_{2} \leq J$ and 
\begin{enumerate}
\item For each $1 \leq j \leq J_{1}$, we have $t_{n}^{j}=0$ and $v^{j}(t):=\NLS(t)\psi^{j}$ does not scatter as $t\to \infty$. 
\item For each $J_{1}+1 \leq j \leq J_{2}$, we have $t_{n}^{j}=0$ and $v^{j}(t)$ scatters as $t\to \infty$. 
\item For each $J_{2}+1 \leq j \leq J$, we have $|t_{n}^{j}|\to \infty$. (In particular, $v^{j}(t)$ scatters in at least one directions by (2) in Lemma \ref{lem4.1}. )
\end{enumerate}
Moreover, $J_{1}\geq 1$. ($J_{2}=J_{1}$ or $J_{2}=J$ may occur. In that case, there is no $j$ satisfying the second and third statement, respectively.)
\end{lemma}

\begin{proof}
Reordering as $t_{n}^{j}=0$ for $1\leq j \leq J_{2}$ and $|t_{n}^{j}|\to \infty$ for $J_{2}+1\leq j \leq J$ is done by suitable numbering. We show $J_{1}\geq 1$.

\textbf{Step 1.} We show that there exists at least one $j$ such that $t_{n}^{j}=0$. Suppose that $|t_{n}^{j}| \to \infty$ for all $j$. By the Pohozaev identity, $E(\varphi_{n})=E(Q)$, and $\|\nabla \varphi_{n}\|_{L^{2}}/\|\nabla Q\|_{L^{2}}>1$, we have
\begin{align*}
	\frac{\|\varphi_{n}\|_{L^{p+1}}^{p+1}}{\|Q\|_{L^{p+1}}^{p+1}}
	=-\frac{1}{2} \frac{E(\varphi_{n})}{E(Q)} + \frac{3}{2} \frac{\|\nabla \varphi_{n}\|_{L^{2}}^{2}}{\|\nabla Q\|_{L^{2}}^{2}}
	\geq -\frac{1}{2} + \frac{3}{2} > 1. 
\end{align*}
Since $|t_{n}^{j}| \to \infty$, we have $\|\NLS(-t_{n}^{j})\psi^{j}\|_{L^{p+1}} \to 0$ as $n \to \infty$ by (2) in Lemma \ref{lem4.1}. Then we have
\begin{align*}
	1\leq \frac{\|\varphi_{n}\|_{L^{p+1}}^{p+1}}{\|Q\|_{L^{p+1}}^{p+1}}
	=\sum_{j=1}^{J}\frac{\|\NLS(-t_{n}^{j})\psi^{j}\|_{L^{p+1}}^{p+1}}{\|Q\|_{L^{p+1}}^{p+1}} + \frac{\|W_{n}^{J}\|_{L^{p+1}}^{p+1}}{\|Q\|_{L^{p+1}}^{p+1}} + o_{n}(1)
\end{align*}
and the right hand side is smaller than 1 for large $n,J$. This is a contradiction. 

\textbf{Step 2.} It remains to prove that there exists at least one $1\leq j \leq J_{2}$ such that $v^{j}(t):=\NLS(t)\psi(t)$ does not scatter in the positive time direction. Suppose the statement fails. Then we have $\lim_{t\to \infty} \|v^{j}(t)\|_{L^{p+1}}=0$ for all $1\leq j \leq J_{2}$ since they scatter. Take $t_{0}$ sufficiently large such that $\|v^{j}(t_{0})\|_{L^{p+1}}^{p+1}\leq \varepsilon /J_{2}$ for all $1\leq j \leq J_{2}$. Let $u_{n}(t):=\NLS(t)\varphi_{n}$. Then the variational structure implies $\|u_{n}(t)\|_{L^{p+1}}/\|Q\|_{L^{p+1}} > 1$ for all $n$ and $t$. 
By the orthogonality \eqref{eq4.8}, we get
\begin{align*}
	\|Q\|_{L^{p+1}}^{p+1} &\leq \|u_{n}(t_{0})\|_{L^{p+1}}^{p+1}
	\\
	& = \sum_{j=1}^{J_{2}} \|v^{j}(t_{0})\|_{L^{p+1}}^{p+1} + \sum_{j=J_{2}+1}^{J} \|v^{j}(t_{0}-t_{n}^{j})\|_{L^{p+1}}^{p+1} +\|W_{n}^{J}(t_{0})\|_{L^{p+1}}^{p+1} +o_{n}(1)
	\\
	& \leq \varepsilon + \sum_{j=J_{2}+1}^{J} \|v^{j}(t_{0}-t_{n}^{j})\|_{L^{p+1}}^{p+1} +\|W_{n}^{J}(t_{0})\|_{L^{p+1}}^{p+1} +o_{n}(1).
\end{align*}
Taking large $J$ and the limit $n\to \infty$, the right hand side goes to $\varepsilon$. Since $\varepsilon$ is sufficiently small, this is a contradiction. 
\end{proof}

%

\section{Proof of Main Theorem}

\subsection{Compactness of the uniformly bounded solution}

%
%

In what follows, let $u_{0}$ satisfy \eqref{MEN} and $u$ be a solution to \eqref{NLS} on at least $[0,\infty)$ satisfying $\mathbf{A}:=\sup_{t>0}\|\nabla u(t)\|_{L^{2}}^{2}<\infty$. In this section, we show that the solution $u$ has a compactness property. 

\begin{proposition}
\label{prop5.1}
Then for any time sequence $\{t_{n}\}$, there exists a subsequence, which is still denoted by $\{t_{n}\}$, a spacial sequence $\{x_{n}\}$, and a function $\psi \in H^{1}(\mathbb{R}^{d})$ such that $u(t_{n}, \cdot + x_{n})$ converges to $\psi$ in $H^{1}(\mathbb{R})$. 
\end{proposition}

\begin{proof}
Apply  Lemma \ref{lem4.1} with $\varphi_{n}=u(t_{n})$,
and reordering Lemma~\ref{lem4.3}.
Since for $j\in [J_1+1,J]$, $v^{j}$ scatters in at least one direction, $E(\psi^{j})=E(v^{j}) \geq 0$. Thus by the energy decoupling we get
\begin{align*}
	\sum_{j=1}^{J_{1}} E(\psi^{j}) \leq E(u(t_{n}))+o_{n}(1)=E(Q)+o_{n}(1)
	\to E(Q).
\end{align*}
Thus there exists $j\in [1,J_{1}]$ such that $E(\psi^{j})\leq E(Q)$. We may assume that $j=1$. We also have
\begin{align*}
	M(\psi^{1})\leq \sum_{j=1}^{J_{1}} M(\psi^{j}) \leq M(u(t_{n}))+o_{n}(1)=M(Q)+o_{n}(1)
	\to M(Q)
\end{align*}
and thus $M(\psi^{1})\leq M(Q)$. 
By Lemma \ref{lem4.3}, $v^{1}$ is a  forward non-scattering solution. 

\textbf{Case 1.} $M(\psi^{1})< M(Q)$ or $E(\psi^{1})< E(Q)$. 

In this case, we have $M(\psi^{1})^{1-s_{c}}E(\psi^{1})^{s_{c}}< M(Q)^{1-s_{c}}E(Q)^{s_{c}}$. 
It follows from \cite{FXC11,AkNa13,Gue14} that $K(\psi^{1})<0$, and that $v^{1}$ blows up or grows up in both directions by \cite{AkNa13,Gue14,DWZ16,Guo16}. There exists $T_{0}>0$ such that $\|\nabla v^{1}(T_{0})\|_{L^{2}}^{2} \geq 2\mathbf{A}$. By \eqref{eq4.7}, we get 
\begin{align*}
	\mathbf{A} \geq \|\nabla u(T_{0}+t_{n})\|_{L^{2}}^{2} 
	&= \sum_{j=1}^{J} \|\nabla v^{j}(T_{0}-t_{n}^{j})\|_{L^{2}}^{2} + \|\nabla W_{n}^{J}(T_{0})\|_{L^{2}}^{2} +o_{n}(1)
	\\
	&\geq \|\nabla v^{1}(T_{0})\|_{L^{2}}^{2} + o_{n}(1)
	\\
	&\geq 2\mathbf{A} +o_{n}(1),
\end{align*}
where we note that $t_{n}^{1}=0$. This is a contradiction.

\textbf{Case 2.} $M(\psi^{1})= M(Q)$ and $E(\psi^{1})= E(Q)$. 


Since we have $M(\psi^{1})=M(Q)$, by the orthogonality of $L^{2}$-norm, we get $\psi^{j}=0$ for all $j\geq 2$. That is, we have $u(t_{n})=\tau_{x_{n}^{1}}\psi^{1} +W_{n}^{1}$. The orthogonality of $L^{2}$-norm also shows that $\|W_{n}^{1}\|_{L^{2}} \to 0$. Moreover, we have
\begin{align*}
	E(Q)=E(u)
	&=E(\psi^{1})+E(W_{n}^{1}) +o_{n}(1)
	\\
	&= E(Q)+E(W_{n}^{1}) +o_{n}(1).
\end{align*}
By this and (4) in Lemma \ref{lem4.1}, it holds that $\|\nabla W_{n}^{1}\|_{L^{2}} \to 0$ as $n \to \infty$. We get 
\begin{align*}
	\|u(t_{n},\cdot+x_{n}^{1}) - \psi^{1}\|_{H^{1}} =\|W_{n}^{1}\|_{H^{1}}\to 0
\end{align*}
as $n\to\infty$. 
\end{proof}

Proposition \ref{prop5.1} implies that there exists $\widetilde{x}:(0,\infty) \to \mathbb{R}^{d}$ such that 
\begin{align*}
	\{u(t,\cdot+\widetilde{x}(t)):t>0\}
\end{align*}
is precompact in $H^{1}(\mathbb{R}^{d})$. This is equivalent to the following: 
for any $\varepsilon>0$, there exists $R=R(\varepsilon)>0$ such that
\begin{align*}
	\int_{|x-\widetilde{x}(t)|>R} |\nabla u(t,x)|^{2} + |u(t,x)|^{2} dx \leq \varepsilon
\end{align*}
for all $t>0$. 

We define $X$ by 
\begin{align*}
	X(t)=
	\begin{cases}
	\widetilde{x}(t), & (t \not\in I_{\mu_{0}}),
	\\
	x(t), & (t \in I_{\mu_{0}}),
	\end{cases}
\end{align*}
where $x$ is the parameter appearing in the modulation argument, Section \ref{sec3}. 
Then, we also find that 
for any $\varepsilon>0$ there exists $R=R(\varepsilon)>0$ such that
\begin{align*}
	\int_{|x-X(t)|>R} |\nabla u(t,x)|^{2} + |u(t,x)|^{2} dx \leq \varepsilon
\end{align*}
for all $t>0$. See the statement below the sketch of the proof of Lemma 6.2 in \cite{DuRo10}. 

\subsection{Finite variance through convergence}

In this section, we show that $u$ is of finite variance. To show this, we first prove that $u$ converges to the ground state. The argument is same as in Section 6.3 in \cite{DuRo10}. See also \cite{CFR20}.

%
%
%
%
%
%

\begin{lemma}
\label{lem5.6}
For small $\varepsilon>0$, there exists a constant $R_{\varepsilon}$ such that 
\begin{align*}
	|A_{R}(u(t))| \cleq \varepsilon |\mu(t)|
\end{align*}
for $t>0$ and $R>R_{\varepsilon}+|X(t)|$. 
\end{lemma}

\begin{proof}
The proof is similar to Step 1 in the proof of Lemma 6.7 in \cite{DuRo10}. 
\end{proof}

\begin{lemma}
\label{lem5.7}
Let $0<t_{1}<t_{2}<\infty$. Then it holds that
\begin{align*}
	\int_{t_{1}}^{t_{2}} |\mu(u(t))| dt \lesssim  (1+\sup_{t_{1} <t < t_{2}}|X(t)|)\{ |\mu(u(t_{1}))|+|\mu(u(t_{2}))|\}.
\end{align*}
\end{lemma}

\begin{proof} 
By Lemma \ref{lem5.6}, there exists a constant $R_{\varepsilon}$ such that if $R\geq R_{\varepsilon}+|X(t)|$, then 
\begin{align*}
	|A_{R}(u(t))|\lesssim \varepsilon |\mu(u(t))|. 
\end{align*}
For $\varepsilon>0$ sufficiently small, 
it follows from this estimate and the localized virial identity that
\begin{align*}
	J_{R}''(t)=4K(u(t)) + A_{R}(u(t))
	\leq c \mu(u(t)) - \varepsilon \mu(u(t)) \lesssim \mu(u(t)) <0
\end{align*}
if $R\geq R_{\varepsilon}+|X(t)|$. 
Let $R=R(t_{1},t_{2}):=R_{\varepsilon}+\sup_{t_{1}\leq t\leq t_{2}}|X(t)|$ for arbitrarily $t_{1},t_{2}$. Then we have
\begin{align*}
	0>\int_{t_{1}}^{t_{2}} \mu(u(t)) dt \gtrsim \int_{t_{1}}^{t_{2}}J_{R}''(t) dt =J'_{R}(t_{2}) - J'_{R}(t_{1}).
\end{align*}
In a similar way to Step 2 of the proof of Lemma 6.7 in \cite{DuRo10}, we have 
\begin{align*}
	|J'_{R}(t)| \lesssim R |\mu(u(t))|.
\end{align*}
Thus we get
\begin{align*}
	\int_{t_{1}}^{t_{2}} \mu(u(t)) dt  \gtrsim - R\{ |\mu(u(t_{1}))|+ |\mu(u(t_{2}))|\}.
\end{align*}
Since $\mu (u(t))<0$, we obtain the desired estimates. 
\end{proof}

\begin{lemma}
\label{lem5.8}
There exists a constant $C>0$ such that 
\begin{align*}
	|X(t_{2}) - X(t_{1})| \leq C \int_{t_{1}}^{t_{2}} |\mu(u(t))| dt
\end{align*}
for any $t_{2},t_{1}$ with $t_{2} \geq t_{1}+1$. 
\end{lemma}

\begin{proof}
See Lemma 6.8 in \cite{DuRo10}. 
\end{proof}

%
%

\begin{lemma}
\label{lem5.10}
There exists a time sequence $\{t_{n}\}$ with $t_{n} \to \infty$ such that $\mu(u(t_{n})) \ceq K(u(t_{n})) \to 0$ as $n \to \infty$. 
\end{lemma} 

\begin{proof}
By \cite{AkNa13,DWZ16}, if there exists $\delta>0$ such that $\sup_{t>0}K(u(t))<-\delta$, then $u$ must grow up in the positive time direction, a contradiction. 
\end{proof}

\begin{lemma}
There exists $\theta_{0}\in\mathbb{R}$, $x_{0}\in\mathbb{R}^{d}$, and $c>0$ such that 
\begin{align*}
	\|u-e^{it+i\theta_{0}} Q(\cdot -x_{0})\|_{H^{1}} \leq e^{-ct}
\end{align*}
for all $t>0$. 
\end{lemma}

\begin{proof}
By Lemmas \ref{lem5.7}--\ref{lem5.10} and the modulation argument in Section \ref{sec3}, we obtain the statement in the same way as in the proof of Proposition 6.1 in \cite{DuRo10}. 
\end{proof}

At last, we prove that $u$ has finite variance. This argument is similar to the argument in Section 5.2 in \cite{DuRo10}, where they show that radial solutions satisfying \eqref{MEN} are of finite variance. We use the convergence to the ground state instead of radial symmetry. Finite variance of $u$ implies that $u=Q^{+}$ and it also gives us Theorem \ref{thm1.1}. 

\begin{lemma}
\label{lem5.12}
The solution $u$ satisfies $\int |x|^{2}|u_{0}(x)|^{2}dx<\infty$.
\end{lemma}

\begin{proof}
%
Let $\varepsilon>0$ be small. For any $R > 0$,
\begin{align*}
	\int_{|x|>R} |\partial_{x}u(t,x)|^{2} + |u(t,x)|^{2} dx 
	&\lesssim \|u(t)-e^{it+i\theta_{0}} Q(\cdot -x_{0})\|_{H^{1}}^{2}
	\\
	&\quad + \int_{|x|>R} |\partial_{x}Q(x-x_{0})|^{2} + |Q(x-x_{0})|^{2} dx
\end{align*}
Thus taking $T_{\varepsilon}>0$ and $R_{\varepsilon}>0$ sufficiently large, for $t > T_\eps$ and $R > R_\eps$ we have
\begin{align*}
	\|u(t)-e^{it+i\theta_{0}} Q(\cdot -x_{0})\|_{H^{1}}^{2} \leq \frac{C^{-1}}{2}\varepsilon
\end{align*}
and 
\begin{align*}
	\int_{|x|>R} |\partial_{x}Q(x-x_{0})|^{2} + |Q(x-x_{0})|^{2} dx
	\leq  \frac{C^{-1}}{2}\varepsilon.
\end{align*}
Therefore, we get
\begin{align*}
	\int_{|x|>R} |\partial_{x}u(t,x)|^{2} + |u(t,x)|^{2} dx 
	\leq \frac{1}{2} \varepsilon + \frac{1}{2} \varepsilon =\varepsilon
\end{align*}
for any $t>T_{\varepsilon}$ and $R>R_{\varepsilon}$. 
By Lemma \ref{lem5.6} as $X(t)\equiv 0$, we get
\begin{align*}
	|A_{R}(u(t))|\lesssim \varepsilon |\mu(u(t))|
\end{align*}
for  $t>T_\eps$ and $R>R_\eps$
(choosing these larger if necessary). 
It follows from this estimate and the localized virial identity that
\begin{align*}
	J_{R}''(t)=4K(u(t)) + A_{R}(u(t))
	\leq c \mu(u(t)) - \varepsilon \mu(u(t)) \lesssim \mu(u(t))
\end{align*}
for $t>T_\varepsilon$, where we note that $\mu(u(t))<0$. 
As a consequence, we find that $J_{R}'(t)>0$ for $t>T_\varepsilon$. Indeed, if not, there exists a time $t_{1}>T_\varepsilon$ such that $J_{R}'(t_{1})\leq 0$. Since $J_{R}''(t)<0$ for $t>T_\varepsilon$, $J_{R}'$ is decreasing and thus there exists $t_{2}>t_{1}$ such that $J_{R}'(t_{2})< 0$ and $J_{R}'(t)\leq  J_{R}'(t_{2})<0$ for $t\geq t_{2}$, so $J_{R}(t)$ becomes negative for large $t$, a contradiction. 
Hence we have $J_{R}'(t)>0$ for $t>T_\varepsilon$. 
It holds that 
\begin{align*}
	\int_{\mathbb{R}^d} \varphi_{R}(x)|u(T_{\varepsilon},x)|^{2}dx = J_{R}(T_{\varepsilon}) \leq J_{R}(t)
	=\int_{\mathbb{R}^d} \varphi_{R}(x)|u(t,x)|^{2}dx
\end{align*}
for all $t>T_{\varepsilon}$. 
Now we have
\begin{align*}
	\int_{\mathbb{R}^d} \varphi_{R}(x)|u(t,x)|^{2}dx 
	&\lesssim R^{2}\|u(t)-e^{it+i\theta_{0}} Q(\cdot -x_{0})\|_{L^{2}}^{2} 
	+\int_{\mathbb{R}^d} |x|^{2}|Q(x -x_{0})|^{2}dx
\end{align*}
and thus
\begin{align*}
	\limsup_{t \to \infty} \int_{\mathbb{R}^d} \varphi_{R}(x)|u(t,x)|^{2}dx
	\lesssim  \int_{\mathbb{R}^d} |x|^{2}|Q(x-x_{0})|^{2}dx
	< C.
\end{align*}
Thus, we get
\begin{align*}
	\int_{\mathbb{R}^d} \varphi_{R}(x)|u(T_{\varepsilon},x)|^{2}dx
	\leq \limsup_{t \to \infty} \int_{\mathbb{R}^d} \varphi_{R}(x)|u(t,x)|^{2}dx
	\leq C.
\end{align*}
Since $\lim_{R \to \infty} \varphi_{R}(x) =|x|^{2}$ for all $x\in \mathbb{R}^d$, it follows from the Fatou lemma that
\begin{align*}
	\int_{\mathbb{R}^d} |x|^{2}|u(T_{\varepsilon},x)|^{2}dx
	\leq \liminf_{R\to \infty} \int_{\mathbb{R}^d} \varphi_{R}(x)|u(T_{\varepsilon},x)|^{2}dx \leq C.
\end{align*}
This completes the proof. 
\end{proof}

\begin{proposition}
We have $u=Q^{+}$ up to symmetries (except for time reversal). 
\end{proposition}

\begin{proof}
By Lemma \ref{lem5.12}, $u$ satisfies \eqref{MEN} and $|x||u_{0}| \in L^{2}$. Thus, by \cite[Theorem 1.3]{CFR20}, $u$ blows up in both time directions or $u=Q^{+}$  up to symmetries. Since $u$ is global in the positive time direction, we get $u=Q^{+}$.
\end{proof}


\section*{Acknowledgement}

Research of the first author is partially supported
by an NSERC Discovery Grant.
The second author is supported by JSPS Overseas Research Fellowship and KAKENHI Grant-in-Aid for Early-Career Scientists No. JP18K13444.


\end{document}